\pdfoutput=1
\documentclass[11pt,letterpaper,twoside,reqno,nosumlimits]{amsart}
\usepackage{graphicx}
\usepackage{amssymb,bm} 
\usepackage{enumerate,setspace}
\usepackage{hyperref}

\renewcommand{\star}{*}

\newcommand{\mat}[1]{\ensuremath{\bm{#1}}} 
\renewcommand{\vec}[1]{\ensuremath{\bm{#1}}}
\newcommand{\e}{\ensuremath{\mathrm{e}}}
\newcommand{\E}{\ensuremath{\mathbb{E}}}
\newcommand{\Prob}[1]{\ensuremath{\mathbb{P}\left\{#1\right\}}}
\newcommand{\lnorm}[2]{\ensuremath{\left\| #2 \right\|_{#1}}}
\newcommand{\R}{\ensuremath{\mathbb{R}}}
\newcommand{\C}{\ensuremath{\mathbb{C}}}
\newcommand{\randcon}{\ensuremath{\Psi}}
\newcommand{\s}{s} 

\newcommand{\norm}[1]{\ensuremath{\left\|#1\right\|}}
\newcommand{\snorm}[1]{\ensuremath{\|#1\|}}
\newcommand{\lambdamax}[1]{\ensuremath{\lambda_{\mathrm{max}}\left(#1\right)}}
\newcommand{\lambdamin}[1]{\ensuremath{\lambda_{\mathrm{min}}\left(#1\right)}}
\newcommand{\Isom}[2]{\ensuremath{\mathbb{V}_{#1}^{#2}}}
\newcommand{\samats}[1]{\ensuremath{\mathbb{M}^{#1}_{\mathrm{sa}}}}
\DeclareMathOperator{\tr}{tr}
\DeclareMathOperator{\rank}{rank}
\newcommand{\trexp}[1]{\ensuremath{\tr\exp\left\{#1\right\}}}

\makeatletter
\renewcommand*\env@cases[1][1.2]{%
  \let\@ifnextchar\new@ifnextchar
  \left\lbrace
  \def\arraystretch{#1}%
  \array{@{}l@{\quad}l@{}}%
}
\makeatother

\newtheorem{thm}{Theorem}
\newtheorem{prop}{Proposition}
\newtheorem{lemma}[thm]{Lemma}
\newtheorem{cor}[thm]{Corollary}

\theoremstyle{remark}
\newtheorem{remark}{Remark}

\numberwithin{equation}{section}
\numberwithin{thm}{section}
\numberwithin{prop}{section}
\numberwithin{defn}{section}
\numberwithin{remark}{section}

\title[Tail Bounds for Eigenvalues of Random Matrices]{Tail Bounds for All Eigenvalues \\ of A Sum of Random Matrices}
\date{July 21, 2011}
\author[A.~Gittens]{Alex~Gittens}
\author[J.~A.~Tropp]{Joel~A.~Tropp}
\thanks{Both authors can be reached at Annenberg Center, MC 305-16, California Institute of Technology, 1200 E. California Blvd., Pasadena, CA 91125. Email: {\tt gittens@cms.caltech.edu} and {\tt jtropp@cms.caltech.edu}. Research supported by ONR awards N00014-08-1-0883 and N00014-11-1-0025, AFOSR award FA9550-09-1-0643, and a Sloan Fellowship.}

\begin{document}
\begin{abstract}
This work introduces the minimax Laplace transform method, a modification of the cumulant-based matrix Laplace transform method developed in \cite{T10a} that yields \emph{both} upper and lower bounds on \emph{each} eigenvalue of a sum of random self-adjoint matrices. This machinery is used to derive eigenvalue analogs of the classical Chernoff, Bennett, and Bernstein bounds.

Two examples demonstrate the efficacy of the minimax Laplace transform. The first concerns the effects of column sparsification on the spectrum of a matrix with orthonormal rows. Here, the behavior of the singular values can be described in terms of coherence-like quantities. The second example addresses the question of relative accuracy in the estimation of eigenvalues of the covariance matrix of a random process. Standard results on the convergence of sample covariance matrices provide bounds on the number of samples needed to obtain relative accuracy in the spectral norm, but these results only guarantee relative accuracy in the estimate of the maximum eigenvalue. The minimax Laplace transform argument establishes that if the lowest eigenvalues decay sufficiently fast, $\Omega(\varepsilon^{-2} \kappa_\ell^2 \ell \log p)$ samples, where $\kappa_\ell = \lambda_1(\mat{C})/\lambda_\ell(\mat{C}),$ are sufficient to ensure that the dominant $\ell$ eigenvalues of the covariance matrix of a $\mathcal{N}(\vec{0}, \mat{C})$ random vector are estimated to within a factor of $1 \pm \varepsilon$ with high probability.

\end{abstract}

\maketitle

\section{Introduction}

The field of nonasymptotic random matrix theory has traditionally focused on the problem of bounding the extreme eigenvalues of a random matrix. In some circumstances, however, we may also be interested in studying the behavior of the interior eigenvalues. In this case, classical tools do not readily apply. Indeed, the interior eigenvalues are determined by the min-max of a random process, which is very challenging to control.

This paper demonstrates that it is possible to combine the matrix Laplace transform method detailed in \cite{T10a} with the Courant--Fischer characterization of eigenvalues to obtain nontrivial bounds on the interior eigenvalues of a sum of random self-adjoint matrices. This approach expands the scope of the matrix probability inequalities from \cite{T10a} so that they provide interesting information about the bulk spectrum.

As one application of our approach, we investigate estimates for the covariance matrix of a centered stationary random process. We show that the eigenvalues of the sample covariance matrix provide relative-error approximations to the eigenvalues of the covariance matrix. We focus on Gaussian processes, but our arguments can be extended to other distributions. The following theorem distills the results in section~\ref{sec:covarianceest}.
\begin{thm}
Let $\mat{C} \in \R^{p \times p}$ be positive semidefinite. Fix an integer $\ell \leq p$ and assume the tail $\{\lambda_i(\mat{C})\}_{i > \ell}$ of the spectrum of $\mat{C}$ decays sufficiently fast that 
\[
\sum_{i > \ell} \lambda_i(\mat{C}) = \mathrm{O}(\lambda_1(\mat{C})).
\]
Let $\{\vec{\eta}_j\}_{j=1}^n \subset \R^p$ be i.i.d.~samples drawn from a $\mathcal{N}(\vec{0}, \mat{C})$ distribution. Define the sample covariance matrix
\[
\widehat{\mat{C}}_n = \frac{1}{n} \sum\nolimits_{j=1}^n \vec{\eta}_j\vec{\eta}_j^\star.
\]
Let $\kappa_\ell$ be the condition number associated with a dominant $\ell$-dimensional invariant subspace of $\mat{C},$
\[
\kappa_\ell = \frac{\lambda_1(\mat{C})}{\lambda_\ell(\mat{C})}.
\]
If $n = \Omega(\varepsilon^{-2} \kappa_\ell^2 \ell \log p),$ then with high probability
\[
|\lambda_k(\widehat{\mat{C}}_n) - \lambda_k(\mat{C})| \leq \varepsilon \lambda_k(\mat{C}) \quad \text{for } k=1, \ldots,\ell.
\]
\label{thm:examplecovarest}
\end{thm}
Thus, assuming sufficiently fast decay of the residual eigenvalues, $n = \Omega(\varepsilon^{-2} \kappa_\ell^2 \ell \log p)$ samples ensure that the top $\ell$ eigenvalues of $\mat{C}$ are captured to relative precision. Spectral decay of this sort is encountered when, e.g., the residual eigenvalues of $\mat{C}$ decay like $k^{-(1+\delta)}$ for some $\delta >0$ or when they arise from measurements corrupted by low-power white noise. 

We contrast Theorem~\ref{thm:examplecovarest} with established spectral norm error bounds for covariance estimation, which do not exploit spectral decay and require that $n = \Omega(\varepsilon^{-2} \kappa_\ell^2 p)$ samples be taken to capture the top $\ell$ eigenvalues to relative precision (see section~\ref{sec:covarianceest}). The estimate in Theorem~\ref{thm:examplecovarest} can be sharpened using information about the spectrum of $\mat{C}$ and the desired failure probability or modified to account for different types of spectral decay. The same tools used in the proof of the theorem can be used to estimate $\lambda_k(\widehat{\mat{C}}_n - \mat{C}).$

\subsection{Related Work}

We believe that this paper contains the first general-purpose tools for studying the full spectrum of a finite-dimensional random matrix. The literature on random matrix theory (RMT) contains some complementary results, but they do not seem to apply with the same generality. Methods from RMT fall into two rough categories: asymptotic methods and nonasymptotic methods. We discuss the relevant results from each in turn.

The modern asymptotic theory began in the 1950s when physicists observed that, on certain scales, the behavior of a quantum system is described by the spectrum of a random matrix \cite{Mehta04}. They further observed the phenomenon of \emph{universality}: as the dimension increases, the spectral statistics become independent of the distribution of the random matrix; instead, they are determined by the symmetries of the distribution \cite{Deift07}. Since these initial observations, physicists, statisticians, engineers, and mathematicians have found manifold applications of the asymptotic theory in high-dimensional statistics \cite{Johnstone01,Johnstone07,ElKaroui08}, physics \cite{Guhretal98,Mehta04}, wireless communication \cite{TulinoVerdu04,SilversteinTulino06}, and pure mathematics \cite{RudnickSarnack96,BerryKeating99}, to mention only a few areas. 

Asymptotic random matrix theory has developed primarily through the examination of specific classes of random matrices. We mention two well-studied classes. Sample covariance matrices take the form $n^{-1} \mat{B}_n \mat{B}_n^\star$, where the columns of $\mat{B}_n$ comprise $n$ independent observations. Wigner matrices are Hermitian matrices whose superdiagonal entries are independent, zero-mean, and have unit variance and whose diagonal entries are i.i.d., real, and have finite variance. 

The fundamental object of study in asymptotic random matrix theory is the empirical spectral distribution function (ESD). Given a random Hermitian matrix $\mat{A}$ of order $n$, its ESD
\[
F^{\mat{A}}(x) = \frac{1}{n} \#\{ 1 \leq i \leq n : \lambda_i(\mat{A}) \leq x \}
\]
is a random distribution function which encodes the statistics of the spectrum of $\mat{A}.$ Wigner's theorem \cite{Wig55}, the seminal result of the asymptotic theory, establishes that if $\{\mat{A}_n\}$ is a sequence of independent, symmetric $n \times n$ matrices with i.i.d.~$\mathcal{N}(0,1)$ entries on and above the diagonal, then the expected ESD of $n^{-1/2} \mat{A}_n$ converges weakly in probability, as $n$ approaches infinity, to the semicircular law given by 
\[
F(x) =  \frac{1}{2\pi}\int_{-\infty}^x \sqrt{4-y^2} \, \mathbf{1}_{[-2,2]}(y)\,{\rm d}y.
\] 
Thus, at least in the limiting sense, the spectra of these random matrices are well characterized. Development of the classical asymptotic theory has been driven by the natural question raised by Wigner's result: to what extent is the semicircular law, and more generally, the existence of a limiting spectral distribution (LSD) universal?

	The literature on the existence and universality of LSDs is massive; we mention only the highlights. It is now known that the semicircular law is universal for Wigner matrices. Suppose that $\{\mat{A}_n\}$ is a sequence of independent $n\times n$ Wigner matrices. Grenander established that if all the moments are finite, then the ESD of $n^{-1/2} \mat{A}_n$ converges weakly to the semicircular law in probability \cite{Grenander63}. Arnold showed that, assuming a finite fourth moment, the ESD almost surely converges weakly to the semicircular law \cite{Arnold71}. Around the same time, Mar\u{c}enko and Pastur determined the form of the limiting spectral distribution of sample covariance matrices \cite{MarcenkoPastur67}. 
	
	More recently, Tao and Vu confirmed the long-conjectured circular law hypothesis. Let $\{\mat{C}_n\}$ be a sequence of independent $n \times n$ matrices whose entries are i.i.d.~and have unit variance. Then the ESD of $n^{-1/2} \mat{C}_n$ converges weakly to the uniform measure on the unit disk, both in probability and almost surely \cite{TaoVu10a}.

Although the convergence rate of the ESD has considerable practical interest, it was not until 1993 that theoretical results became available when Bai showed that for Wigner matrices \cite{Bai93a} and sample covariance matrices \cite{Bai93b} the expected ESDs of $n^{-1/2} \mat{A}_n$ and $n^{-1} \mat{B}_n \mat{B}_n^\star,$ respectively, both converge pointwise at a rate of O($n^{-1/4}$). Later, Bai and coauthors established the pointwise convergence in probability of the ESD of the normalized Wigner matrix $n^{-1/2}\mat{A}_n$ \cite{Bai97} and greatly improved the convergence rates \cite{BaiWigner99,Bai02,Bai03}. The strongest result to date is due to Bai et al., who have shown that, if the entries of the Wigner matrix possess finite sixth moments, then pointwise convergence in probability of the ESD of $n^{-1/2} \mat{A}_n$ occurs at the rate of O($n^{-1/2}$) \cite{Bai11}. 

Classically, individual eigenvalues have been studied through the limiting behavior of the extremal eigenvalues and the asymptotic joint distribution of several eigenvalues. Much is known about the limiting distribution of the largest eigenvalues of Wigner and covariance matrices. Geman showed that if the columns of $\mat{B}_n$ are drawn from a sufficiently regular distribution, then the largest eigenvalue of the sample covariance matrix $n^{-1} \mat{B}_n \mat{B}_n^\star$ converges almost surely to a limit \cite{Geman80}. Bai, Yin, and coauthors showed that the existence of a fourth moment is both necessary and sufficient for the existence of such a limit \cite{YinBaiKrishnaiah88,BaiSilversteinYin88}. They also identified necessary and sufficient conditions for the existence of limits for the smallest and largest eigenvalues of a normalized Wigner matrix $n^{-1/2} \mat{A}_n$ \cite{BY88b}. El Karoui has recently described the limiting behavior of the leading eigenvalues of a large class of sample covariance matrices \cite{ElKaroui07}. 

Less is known about the rate of convergence of the eigenvalues, but some results are available. Write the eigenvalues of a self-adjoint matrix $\mat{A}$ in nonincreasing order $\lambda_1 \geq \ldots \geq \lambda_n.$  For $1 \leq j \leq n,$ the classical location $\gamma_j$ of the $j$th eigenvalue of the normalized Wigner matrix $n^{-1/2} \mat{A}_n$ is defined via the relation
\[
\int_{-\infty}^{\gamma_j} \rho_{sc}(x) \, dx = \frac{j}{n},
\]
where $\rho_{sc}$ is the density associated with the semicircular law. Intuitively, the facts that $F^{\frac{1}{\sqrt{n}}\mat{A}_n} \rightarrow F^{sc}$ and $F^{\frac{1}{\sqrt{n}}\mat{A}_n}(\lambda_j) = j/n$ suggest that $\frac{1}{\sqrt{n}} \lambda_j \rightarrow \gamma_j.$ Indeed, it follows from~\cite{BY88a,BY88b} that
\[
\lambda_j = \sqrt{n} \gamma_j + {\rm o}(\sqrt{n})
\]
asymptotically almost surely. Under the assumption that the entries exhibit uniform subgaussian decay, Erd\"os, Yau, and Yin have strengthened this result by showing that, up to log factors, the eigenvalues of $n^{-1/2} \mat{A}_n$ are within O($n^{-2/3}$) of their classical position with high probability \cite{ErdosYauYin10}. More generally, Tao and Vu have established the universality of a result due to Gustavsson \cite{Gustavsson05} in the complex Gaussian Wigner case: $ (\log n)^{-1/2}( \sqrt{n} \lambda_j - n \gamma_j)$ is asymptotically normally distributed \cite{TaoVu11}. Further, they have shown that eigenvalues in the bulk of the spectrum ($ j = \Omega(n)$) of a Wigner matrix satisfy
\[
\E |\lambda_j - \sqrt{n} \gamma_j|^2 = \mathrm{O}(n^{-c}),
\]
for some universal constant $c >0$ \cite{TaoVu10}.

In contrast to the asymptotic theory, which remains to a large extent driven by the study of particular classes of random matrices, the nonasymptotic theory has developed as a collection of techniques for addressing the behavior of a broad range of random matrices. The nonasymptotic theory has its roots in geometric functional analysis in the 1970s, where random matrices were used to investigate the local properties of Banach spaces \cite{LindenstraussMilman93,DavidsonSzarek01,VershyninRMTnotes}. Since then, the nonasymptotic theory has found applications in areas including theoretical computer science \cite{Achlioptas03,Vempala04,SpielmanSrivastava08}, machine learning \cite{DM05}, optimization \cite{Nem07,So09}, and numerical linear algebra \cite{DM10,Halkoetal11,Mahoney11}. 

As is the case in the asymptotic theory, the sharpest and most comprehensive results available in the nonasymptotic theory concern the behavior of Gaussian matrices. The amenability of the Gaussian distribution makes it possible to obtain results such as Szarek's nonasymptotic analog of the Wigner semicircle theorem for Gaussian matrices \cite{Sza90} and Chen and Dongarra's bounds on the condition number of Gaussian matrices \cite{ChenDongarra05}.
The properties of less well-behaved random matrices can sometimes be related back to those of Gaussian matrices using probabilistic tools, such as symmetrization; see, e.g., the derivation of Lata{\l}a's bound on the norms of zero-mean random matrices \cite{Lat04}. 

More generally, bounds on extremal eigenvalues can be obtained from knowledge of the moments of the entries. For example, the smallest singular value of a square matrix with i.i.d.~zero-mean subgaussian entries with unit variance is O($n^{-1/2}$) with high probability \cite{RV08}. Concentration of measure results, such as Talagrand's concentration inequality for product spaces \cite{Talagrand95}, have also contributed greatly to the nonasymptotic theory. We mention in particular the work of Achlioptas and McSherry on randomized sparsification of matrices \cite{AM01,AM07}, that of Meckes on the norms of random matrices \cite{Meckes04}, and that of Alon, Krivelevich and Vu \cite{AlonKrivelevichVu02} on the concentration of the largest eigenvalues of random symmetric matrices, all of which are applications of Talagrand's inequality. In cases where geometric information on the distribution of the random matrices is available, the tools of empirical process theory---such as the generic chaining, also due to Talagrand \cite{Talagrand05}---can be used to convert this geometric information into information on the spectra. One natural example of such a case consists of matrices whose rows are independently drawn from a log-concave distribution \cite{MendelsonPajor06,ALPT10b}.

The noncommutative Khintchine inequality (NCKI), which bounds the moments of the norm of a sum of fixed matrices modulated by random signs \cite{Lust-Piquard86,Lust-PiquardPisier91}, is a widely used tool in the nonasymptotic theory. Despite its power, the NCKI is unwieldy. To use it, one must reduce the problem to a suitable form by applying symmetrization and decoupling arguments and exploiting the equivalence between moments and tail bounds. It is often more convenient to apply the NCKI in the guise of a lemma, due to Rudelson \cite{RU99}, that provides an analog of the law of large numbers for sums of rank-one matrices. This result has found many applications, including column-subset selection \cite{RV07} and the fast approximate solution of least-squares problems \cite{DMMS11}. The NCKI and its corollaries do not always yield sharp results because parasitic logarithmic factors arise in many settings.

The current paper is ultimately based on the influential work of Ahlswede and Winter~\cite{AW02}. This line of research leads to explicit tail bounds for the maximum eigenvalue of a sum of random matrices. These probability inequalities parallel the classical scalar tail bounds due to Bernstein and others. Matrix probability inequalities allow us to obtain valuable information about the maximum eigenvalue of a random matrix with very little effort. Furthermore, they apply to a wide variety of random matrices. We note, however, that matrix probability inequalities can lead to parasitic logarithmic factors similar to those that emerge from the NCKI.

Major contributions to the literature on matrix probability inequalities include the papers \cite{CM08,Recht09,Gross11}. We emphasize two works of Oliveira \cite{Oliv09,Oliv10} that go well beyond earlier research. The sharpest current results appear in the works of Tropp \cite{T10a,T10b,Tropp11}. Recently, Hsu, Kakade, and Zhang \cite{HsuKakadeZhang11} have modified Tropp's approach to establish matrix probability inequalities that depend on an intrinsic dimension parameter, rather than the ambient dimension.

\subsection{Outline}
In section \ref{sec:background}, we introduce the notation used in this paper and state a convenient version of the Courant--Fischer theorem. In section \ref{sec:laplacetransform}, we use the Courant--Fischer theorem to extend the Laplace transform technique from \cite{T10a} to apply to all the eigenvalues of self-adjoint matrices, thereby obtaining the minimax Laplace transform. We apply this technique in sections \ref{sec:chernoffbounds} and \ref{sec:bernsteinbounds} to develop eigenvalue analogs of the classical Chernoff and Bernstein bounds. The final two sections illustrate, using two familiar problems, that the minimax Laplace technique gives us significantly more information on the spectra of random matrices than current approaches. In section \ref{sec:colsubsampling}, we use the Chernoff bounds to quantify the effects of column sparsification on all the singular values of matrices with orthogonal rows. In section \ref{sec:covarianceest}, we consider the question of how fast, in relative error, the eigenvalues of empirical covariance matrices converge.

\section{Background and Notation}
\label{sec:background}
We establish the notation used in the sequel and state a convenient version of the Courant--Fischer theorem.

Unless otherwise stated, we work over the complex field. The $k$th column of the matrix $\mat{A}$ is denoted by $\mat{a}_k,$ and the entries are denoted $a_{jk}$ or $(\mat{A})_{jk}.$ We define $\samats{n}$ to be the set of self-adjoint matrices with dimension $n.$ The eigenvalues of a matrix $\mat{A}$ in $\samats{n}$ are arranged in weakly decreasing order: $\lambdamax{\mat{A}} = \lambda_1(\mat{A}) \geq \lambda_2(\mat{A}) \geq \cdots \geq \lambda_n(\mat{A}) = \lambdamin{\mat{A}}.$ Likewise, singular values of a rectangular matrix $\mat{B}$ with rank $r$ are ordered $\s_1(\mat{B}) \geq \s_2(\mat{B}) \geq \cdots \geq \s_r(\mat{B}).$ The spectral norm of a matrix $\mat{B}$ is expressed as $\snorm{\mat{B}}.$ We often compare self-adjoint matrices using the semidefinite ordering. In this ordering, $\mat{A}$ is greater than or equal to $\mat{B}$, written $\mat{A} \succeq \mat{B}$ or $\mat{B} \preceq \mat{A},$ when $\mat{A} - \mat{B}$ is positive semidefinite.

The expectation of a random variable is denoted by $\E X.$ We write $X \sim \text{Bern}(p)$ to indicate that $X$ has a Bernoulli distribution with mean $p.$

One of our central tools is the variational characterization of the eigenvalues of a self-adjoint matrix given by the Courant--Fischer theorem. 
For integers $d$ and $n$ satisfying $1 \leq d \leq n$, the complex Stiefel manifold 
\[
\Isom{d}{n} = \{\mat{V} \in \C^{n \times d} \,:\, \mat{V}^\star \mat{V} = \mathbf{I} \}
\]
is the collection of orthonormal bases for the $d$-dimensional subspaces of $\C^n,$ or, equivalently, the collection of all isometric embeddings of $\C^d$ into $\C^n.$ Let $\mat{A}$ be a self-adjoint matrix with dimension $n,$ and let $\mat{V} \in \Isom{d}{n}$ be an orthonormal basis for a subspace of $\C^n.$ Then the matrix $\mat{V}^\star \mat{A} \mat{V}$ can be interpreted as the compression of $\mat{A}$ to the space spanned by $\mat{V}.$

\begin{prop}[Courant--Fischer]
\label{prop:isometrycf}
Let $\mat{A}$ be a self-adjoint matrix with dimension $n$. Then
\begin{align}
\lambda_k(\mat{A}) & = \min_{\mat{V} \in \Isom{n-k+1}{n}} \lambdamax{\mat{V}^\star\mat{AV}} \quad \text{and}
\label{eqn:maxvariation} \\
\lambda_k(\mat{A}) & = \max_{\mat{V} \in \Isom{k}{n}} \lambdamin{\mat{V}^\star \mat{AV}}. \label{eqn:minvariation}
\end{align}
A matrix $\mat{V}_- \in \Isom{k}{n}$ achieves equality in~\eqref{eqn:minvariation} if and only if its columns span a dominant $k$-dimensional invariant subspace of $\mat{A}.$ Likewise, a matrix $\mat{V}_+ \in \Isom{n-k+1}{n}$ achieves equality in~\eqref{eqn:maxvariation} if and only if its columns span a bottom $(n-k+1)$-dimensional invariant subspace of $\mat{A}$.
\end{prop}

The $\pm$ subscripts in Proposition \ref{prop:isometrycf} are chosen to reflect the fact that $\lambda_k(\mat{A})$ is the \emph{minimum} eigenvalue of $\mat{V}_-^\star\mat{A}\mat{V}_-$ and the \emph{maximum} eigenvalue of $\mat{V}_+^\star \mat{A} \mat{V}_+.$ 
As a consequence of Proposition \ref{prop:isometrycf}, when $\mat{A}$ is self-adjoint, \mbox{$\lambda_k(-\mat{A}) = -\lambda_{n-k+1}(\mat{A}).$} This fact allows us to use the same techniques we develop for bounding the eigenvalues from above to bound them from below.

\section{ Tail Bounds For Interior Eigenvalues}
\label{sec:laplacetransform}
In this section we develop a generic bound on the tail probabilities of eigenvalues of sums of independent, random, self-adjoint matrices. We establish this bound by supplementing the matrix Laplace transform methodology of \cite{T10a} with Proposition \ref{prop:isometrycf} and a new result, due to Lieb and Seiringer \cite{LS05}, on the concavity of a certain trace function on the cone of positive-definite matrices.

 First we observe that the Courant--Fischer theorem allows us relate the behavior of the $k$th eigenvalue of a matrix to the behavior of the largest eigenvalue of an appropriate compression of the matrix.

\begin{thm}
Let $\mat{X}$ be a random self-adjoint matrix with dimension $n,$ and let $k \leq n$ be an integer. Then, for all $t \in \R,$
\begin{equation}
\Prob{\lambda_k(\mat{X}) \geq t} \leq \inf_{\theta > 0} \min_{\mat{V} \in \Isom{n-k+1}{n}} \left\{ \e^{-\theta t} \cdot \E\tr\e^{\theta \mat{V}^\star\mat{XV}} \right\}.
\label{eqn:laplacetform}
\end{equation}
\label{thm:laplacetform}
\end{thm}

\begin{proof}
Let $\theta$ be a fixed positive number. Then
\begin{multline*}
\Prob{\lambda_k(\mat{X}) \geq t}  = 
\Prob{\lambda_k(\theta \mat{X}) \geq \theta t} = 
\Prob{\e^{\lambda_k(\theta \mat{X})} \geq \e^{\theta t}} \\
\leq  \e^{-\theta t} \cdot \E \e^{\lambda_k(\theta \mat{X})} = 
\e^{-\theta t} \cdot \E \exp\left\{\min_{\mat{V} \in \Isom{n-k+1}{n}} \lambdamax{\theta \mat{V}^\star\mat{XV}}\right\}. 
\end{multline*}
The first identity follows from the positive homogeneity of eigenvalue maps and the second from the monotonicity of the scalar exponential function. The final two relations are Markov's inequality and \eqref{eqn:maxvariation}.

To continue, we need to bound the expectation. Interchange the order of the exponential and the minimum; then apply the spectral mapping theorem to see that
\begin{align*}
 \E \exp\bigg\{\min_{\mat{V} \in \Isom{n-k+1}{n}} \lambdamax{\theta \mat{V}^\star \mat{XV}} \bigg\} & = \E \min_{\mat{V} \in \Isom{n-k+1}{n}} \lambdamax{\exp(\theta \mat{V}^\star\mat{XV})} \\
  & \leq \min_{\mat{V} \in \Isom{n-k+1}{n}} \E \lambdamax{\exp(\theta \mat{V}^\star\mat{XV})} \\
  & \leq \min_{\mat{V} \in \Isom{n-k+1}{n}} \E \tr \exp(\theta \mat{V}^\star\mat{XV}).
\end{align*}
The first inequality is Jensen's. The second inequality follows because the exponential of a self-adjoint matrix is positive definite, so its largest eigenvalue is smaller than its trace.

Combine these observations and take the infimum over all positive $\theta$ to complete the argument.
\end{proof}

We are interested in the case where the matrix $\mat{X}$ in Theorem \ref{thm:laplacetform} can be expressed as a sum of independent random matrices. In this case, we use the following result to develop the right-hand side of the Laplace transform bound \eqref{eqn:laplacetform}.



\begin{thm}
Consider a finite sequence $\{\mat{X}_j\}$ of independent, random, self-adjoint matrices with dimension $n$ and a sequence $\{\mat{A}_j\}$ of fixed self-adjoint matrices with dimension $n$ that satisfy the relations
\begin{equation}
 \E\e^{\mat{X}_j} \preceq \e^{\mat{A}_j}. 
 \label{eqn:logmgfdomination}
\end{equation}
Let $\mat{V} \in \Isom{k}{n}$ be an isometric embedding of $\C^k$ into $\C^n$ for some $k \leq n.$ Then
\begin{equation}
 \E \tr \exp\left\{\sum\nolimits_j \mat{V}^\star \mat{X}_j \mat{V} \right\} \leq \tr \exp\left\{\sum\nolimits_j \mat{V}^\star \mat{A}_j \mat{V} \right\}.
 \label{eqn:pinchedmgf}
\end{equation}
In particular, 
\begin{equation}
 \E \tr \exp\left\{\sum\nolimits_j \mat{X}_j \right\} \leq \tr \exp\left\{\sum\nolimits_j \mat{A}_j \right\}.
 \label{eqn:unpinchedmgf}
\end{equation}

\label{thm:mgf}
\end{thm}





Theorem \ref{thm:mgf} is an extension of Lemma 3.4 of \cite{T10a}, which establishes the special case \eqref{eqn:unpinchedmgf}. The proof depends upon a recent result due to Lieb and Seiringer \cite[Thm.~3]{LS05} that extends Lieb's earlier result \cite[Thm.~6]{Lieb1973}.

\begin{prop}[Lieb--Seiringer 2005]
Let $\mat{H}$ be a self-adjoint matrix with dimension $k.$ Let $\mat{V} \in \Isom{k}{n}$ be an isometric embedding of $\C^k$ into $\C^n$ for some $k \leq n.$ Then the function
\begin{equation*}
 \mat{A} \longmapsto \trexp{\mat{H} + \mat{V}^\star (\log\mat{A}) \mat{V}}
\end{equation*}
is concave on the cone of positive-definite matrices in \samats{n}.
 \label{prop:pinchedtracefunctional}
\end{prop}

\begin{proof}[Proof of Theorem \ref{thm:mgf}]
 First, note that \eqref{eqn:logmgfdomination} and the operator monotonicity of the matrix logarithm yield the following inequality for each $k$: 
\begin{equation}
\log \E \e^{\mat{X}_k} \preceq \mat{A}_k.
\label{eqn:logdom}
\end{equation}
Let $\E_k$ denote expectation conditioned on the first $k$ summands, $\mat{X}_1$ through $\mat{X}_k.$ Then
\begin{align*}
 \E \trexp{\sum_{j \leq \ell} \mat{V}^\star \mat{X}_j \mat{V}} & = \E\E_1\cdots \E_{\ell-1} \trexp{\sum_{j \leq \ell-1} \mat{V}^\star \mat{X}_j \mat{V} + \mat{V}^\star \left( \log \e^{\mat{X}_\ell} \right) \mat{V}} \\
 & \leq \E\E_1\cdots \E_{\ell-2} \trexp{\sum_{j \leq \ell-1} \mat{V}^\star \mat{X}_j \mat{V} + \mat{V}^\star \left( \log\E\e^{\mat{X}_\ell} \right) \mat{V}} \\
 & \leq \E\E_1\cdots \E_{\ell-2} \trexp{\sum_{j \leq \ell-1} \mat{V}^\star \mat{X}_j \mat{V} + \mat{V}^\star \left( \log\e^{\mat{A}_\ell} \right) \mat{V} } \\
 & = \E\E_1\cdots \E_{\ell-2} \trexp{\sum_{j \leq \ell-1} \mat{V}^\star \mat{X}_j \mat{V} + \mat{V}^\star \mat{A}_\ell \mat{V} }.
\end{align*}
The first inequality follows from Proposition \ref{prop:pinchedtracefunctional} and Jensen's inequality, and the second depends on \eqref{eqn:logdom} and the monotonicity of the trace exponential. Iterate this argument to complete the proof.
\end{proof}

Our main result follows from combining Theorem \ref{thm:laplacetform} and Theorem \ref{thm:mgf}.

\begin{thm}[Minimax Laplace Transform]
Consider a finite sequence $\{\mat{X}_j\}$ of independent, random, self-adjoint matrices with dimension $n$, and let $k\leq n$ be an integer.
\begin{enumerate}[(i)]
 \item Let $\{\mat{A}_j\}$ be a sequence of self-adjoint matrices that satisfy the semidefinite relations
\[
 \E\e^{\theta \mat{X}_j} \preceq \e^{g(\theta) \mat{A}_j}
\]
where $g : (0,\infty) \rightarrow [0, \infty).$ Then, for all $t \in \R,$ 
\[
 \Prob{\lambda_k\left(\sum\nolimits_j \mat{X}_j \right) \geq t } \leq \inf_{\theta > 0}\; \min_{\mat{V} \in \Isom{n-k+1}{n}} \bigg[ \e^{-\theta t} \cdot \tr \exp\left\{ g(\theta) \sum\nolimits_j \mat{V}^\star \mat{A}_j \mat{V}\right\}\bigg].
\]
 \label{eqn:uncompressedeigtails}
\item Let $\{\mat{A}_j:\Isom{n-k+1}{n} \rightarrow \samats{n} \}$ be a sequence of functions that satisfy the semidefinite relations 
\[
 \E\e^{\theta\mat{V}^\star \mat{X}_j \mat{V}} \preceq \e^{g(\theta) \mat{A}_j(\mat{V})}
\]
for all $\mat{V} \in \Isom{n-k+1}{n},$ where $g : (0, \infty) \rightarrow [0, \infty).$ Then, for all~$t \in \R,$
\[
 \Prob{\lambda_k\left(\sum\nolimits_j \mat{X}_j \right) \geq t } \leq \inf_{\theta > 0 } \; \min_{\mat{V} \in \Isom{n-k+1}{n}} \bigg[ \e^{-\theta t} \cdot \tr \exp\left\{ g(\theta) \sum\nolimits_j \mat{A}_j(\mat{V}) \right\}\bigg].
\]
 \label{eqn:compressedeigtails}
\end{enumerate}

\label{thm:eigtails}
\end{thm}

The first bound in Theorem \ref{thm:eigtails} requires less detailed information on how compression affects the summands but correspondingly does not give as sharp results as the second. 

In the following two sections, we use the minimax Laplace transform method to derive Chernoff and Bernstein inequalities for the interior eigenvalues of a sum of independent random matrices. Tail bounds for the eigenvalues of matrix Rademacher and Gaussian series, eigenvalue Hoeffding, and matrix martingale eigenvalue tail bounds can all be derived in a similar manner; see \cite{T10a} for relevant details.

\section{ Chernoff bounds}
\label{sec:chernoffbounds}

Classical Chernoff bounds establish that the tails of a sum of independent nonnegative random variables decay subexponentially. \cite{T10a} develops Chernoff bounds for the maximum and minimum eigenvalues of a sum of independent positive-semidefinite matrices. We extend this analysis to study the interior eigenvalues. 

Intuitively, the eigenvalue tail bounds should depend on how concentrated the summands are; e.g., the maximum eigenvalue of a sum of operators whose ranges are aligned is likely to vary more than that of a sum of operators whose ranges are orthogonal. To measure how much a finite sequence of random summands $\{\mat{X}_j\}$ concentrates in a given subspace, we define a function $\randcon : \bigcup_{1 \leq k \leq n} \Isom{k}{n} \rightarrow \R$ that satisfies
\begin{equation}
	\max\nolimits_j \lambdamax{\mat{V}^\star \mat{X}_j \mat{V}} \leq \randcon(\mat{V}) \qquad \text{ almost surely for each } \mat{V} \in \bigcup_{1 \leq k \leq n} \Isom{k}{n}. 
\label{eqn:randcondef}
\end{equation}
The sequence $\{\mat{X}_j\}$ associated with $\randcon$ will always be clear from context. We have the following result.
\begin{thm}[Eigenvalue Chernoff Bounds]
Consider a finite sequence $\{\mat{X}_j\}$ of independent, random, positive-semidefinite matrices with dimension $n.$ Given an integer $k \leq n$, define 
\[
\mu_k = \lambda_k\left(\sum\nolimits_j \E \mat{X}_j\right),
\]
and let $\mat{V}_{+} \in \Isom{n-k+1}{n}$ and $\mat{V}_{-} \in \Isom{k}{n}$ be isometric embeddings that satisfy 
 $$ \mu_k = \lambdamax{\sum\nolimits_j \mat{V}_{+}^\star (\E \mat{X}_j)\mat{V}_{+}} = \lambdamin{\sum\nolimits_j \mat{V}_{-}^\star (\E \mat{X}_j)\mat{V}_{-}}. $$

Then 
\begin{align*}
\Prob{\lambda_k\left( \sum\nolimits_j \mat{X}_j \right) \geq (1+\delta)\mu_k} & \leq (n-k+1) \cdot \left[\frac{\e^\delta}{(1+\delta)^{1+\delta}} \right]^{\mu_k/\randcon(\mat{V}_{+})} 
 & \text{for } \delta > 0, \text{ and} \\ 
\Prob{\lambda_k\left(\sum\nolimits_j \mat{X}_j \right) \leq (1-\delta)\mu_k} & \leq k \cdot \left[\frac{\e^{-\delta}}{(1-\delta)^{1-\delta}}\right]^{\mu_k/\randcon(\mat{V}_{-})} &  \text{for } \delta \in [0,1),
\end{align*}
where $\randcon$ is a function that satisfies \eqref{eqn:randcondef}.
\label{thm:chernoff}
\end{thm}

Theorem \ref{thm:chernoff} tells us how the tails of the $k$th eigenvalue are controlled by the variation of the random summands in the top and bottom invariant subspaces of $\sum_j \E \mat{X}_j.$ Up to the dimensional factors $k$ and $n-k+1$, the eigenvalues exhibit binomial-type tails. When $k=1$ (respectively, $k=n$) Theorem \ref{thm:chernoff} controls the probability that the largest eigenvalue of the sum is small (respectively, the probability that the smallest eigenvalue of the sum is large), thereby complementing the one-sided Chernoff bounds of \cite{T10a}.


\begin{remark}
If it is difficult to estimate $\randcon(\mat{V}_+)$ or $\randcon(\mat{V}_{-}),$ one can resort to the weaker estimates
\begin{align*}
	\randcon(\mat{V}_{+}) & \leq \max_{\mat{V} \in \Isom{n-k+1}{n}} \max\nolimits_j \norm{\mat{V}^\star \mat{X}_j \mat{V}} = \max\nolimits_j \norm{\mat{X}_j} \\
	\randcon(\mat{V}_{-}) & \leq \max_{\mat{V} \in \Isom{k}{n}} \max\nolimits_j \norm{\mat{V}^\star \mat{X}_j \mat{V}} = \max\nolimits_j \norm{\mat{X}_j}.
\end{align*}
\end{remark}

Theorem \ref{thm:chernoff} follows from Theorem \ref{thm:eigtails} using an appropriate bound on the matrix moment generating functions. The following lemma is due to Ahlswede and Winter \cite{AW02}; see also \cite[Lem.~5.8]{T10a}.

\begin{lemma}
Suppose that $\mat{X}$ is a random positive-semidefinite matrix that satisfies $\lambdamax{\mat{X}} \leq 1.$ Then
$$ \E \e^{\theta \mat{X}} \preceq \exp\left( (\e^\theta - 1) (\E\mat{X}) \right) \quad \text{ for } \theta \in \R. $$
\label{lemma:chernoffmgf}
\end{lemma}

\begin{proof}[Proof of Theorem \ref{thm:chernoff}, upper bound]
We consider the case where $\randcon(\mat{V}_{+}) = 1;$ the general case follows by homogeneity.
Define 
$$ \mat{A}_j(\mat{V}_{+}) = \mat{V}_{+}^\star (\E \mat{X}_j) \mat{V}_{+} \quad \text{and}\quad  g(\theta) = \e^\theta - 1. $$
Theorem \ref{thm:eigtails}\eqref{eqn:compressedeigtails} and Lemma \ref{lemma:chernoffmgf} imply that 
\[
\Prob{\lambda_k\left(\sum\nolimits_j \mat{X}_j\right) \geq (1+\delta)\mu_k} \leq \inf_{\theta > 0} \e^{-\theta (1 + \delta)\mu_k} \cdot \tr\exp \left\{g(\theta) \sum\nolimits_j\mat{V}_{+}^\star (\E \mat{X}_j) \mat{V}_{+} \right\}. 
\]
Bound the trace by the maximum eigenvalue, taking into account the reduced dimension of the summands:
\begin{align*}
\tr\exp \left\{g(\theta) \sum\nolimits_j \mat{V}_{+}^\star(\E \mat{X}_j)\mat{V}_{+}\right\} & \leq (n-k+1)\cdot  \lambdamax{\exp\left\{g(\theta) \sum\nolimits_j \mat{V}_{+}^\star(\E \mat{X}_j)\mat{V}_{+}\right\}} \\
 & = (n-k+1) \cdot \exp\left\{g(\theta) \cdot \lambdamax{\sum\nolimits_j\mat{V}_{+}^\star(\E\mat{X}_j)\mat{V}_{+}}\right\}.
\end{align*}
The equality follows from the spectral mapping theorem. Identify the quantity $\mu_k$; then combine the last two inequalities to obtain
\[
 \Prob{\lambda_k\left(\sum\nolimits_j \mat{X}_j \right) \geq (1 + \delta)\mu_k} \leq \\
(n-k+1) \cdot \inf_{\theta >0 } \e^{[g(\theta) -\theta(1 + \delta)] \mu_k }. 
\]
The right-hand side is minimized when $\theta = \log(1+\delta),$ which gives the desired upper tail bound.
\end{proof}

\begin{proof}[Proof of Theorem \ref{thm:chernoff}, lower bound]
As before, we consider the case where $\randcon(\mat{V}_{-}) = 1.$ Clearly,
\begin{equation}
 \Prob{ \lambda_k\left(\sum\nolimits_j \mat{X}_j \right) \leq (1-\delta)\mu_k } = \Prob{ \lambda_{n-k+1} \left(\sum\nolimits_j -\mat{X}_j \right) \geq -(1-\delta)\mu_k }.
 \label{eqn:chernoffeqn1}
\end{equation}
Apply Lemma \ref{lemma:chernoffmgf} to see that, for $\theta > 0,$
\[
\E \e^{\theta (-\mat{V}^\star_- \mat{X}_j \mat{V}_-)} = \E \e^{(-\theta)\mat{V}^\star_- \mat{X}_j \mat{V}_-} \preceq \exp\big(g(\theta)\cdot \mat{V}^\star_-(-\E\mat{X}_j)\mat{V}_-\big), 
\]
where $g(\theta) = 1-\e^{-\theta}.$  
Theorem \ref{thm:eigtails}\eqref{eqn:compressedeigtails} thus implies that the latter probability in \eqref{eqn:chernoffeqn1} is bounded by
\[
  \inf_{\theta > 0} \e^{\theta(1-\delta)\mu_k}\cdot \trexp{ g(\theta) \sum\nolimits_j\mat{V}_{-}^\star (- \E\mat{X}_j) \mat{V}_{-} }.
\]
Using reasoning analogous to that in the proof of the upper bound, we justify the first of the following inequalities:
\begin{align*}
\trexp{ g(\theta) \sum\nolimits_j\mat{V}_{-}^\star (-\E \mat{X}_j) \mat{V}_{-}} & \leq k \cdot \exp\left\{ \lambdamax{ g(\theta)  \sum\nolimits_j\mat{V}_{-}^\star (-\E \mat{X}_j) \mat{V}_{-}} \right\} \\
& = k \cdot \exp\left\{ -g(\theta) \cdot \lambdamin{\sum\nolimits_j \mat{V}_{-}^\star (\E \mat{X}_j) \mat{V}_{-}} \right\} \\
& = k \cdot \exp\left\{ -g(\theta) \mu_k \right\}.
\end{align*}
The remaining equalities follow from the fact that $-g(\theta) <0$ and the definition of $\mu_k.$

This argument establishes the bound
\[
\Prob{\lambda_k\left(\sum\nolimits_j \mat{X}_j \right) \leq (1 - \delta)\mu_k} \leq k \cdot \inf_{\theta > 0} \e^{[\theta(1 - \delta) - g(\theta)]\mu_k}.
\] 
The right-hand side is minimized when $\theta = -\log(1-\delta),$ which gives the desired lower tail bound.
\end{proof}

\section{Bennett and Bernstein inequalities}
\label{sec:bernsteinbounds}

The classical Bennett and Bernstein inequalities use the variance or knowledge of the moments of the summands to control the probability that a sum of independent random variables deviates from its mean. In \cite{T10a}, matrix Bennett and Bernstein inequalities are developed for the extreme eigenvalues of self-adjoint random matrix sums. We establish that the interior eigenvalues satisfy analogous inequalities.

As in the derivation of the Chernoff inequalities of section \ref{sec:chernoffbounds}, we need a measure of how concentrated the random summands are in a given subspace. Recall that the function $\randcon : \bigcup_{1 \leq k \leq n} \Isom{k}{n} \rightarrow \R$ satisfies
\begin{equation}
	\max\nolimits_j \lambdamax{\mat{V}^\star \mat{X}_j \mat{V}} \leq \randcon(\mat{V}) \qquad \text{ almost surely for each } \mat{V} \in \bigcup_{1 \leq k \leq n} \Isom{k}{n}. 
\label{eqn:randcondef2}
\end{equation}
The sequence $\{\mat{X}_j\}$ associated with $\randcon$ will always be clear from context.

\begin{thm}[Eigenvalue Bennett Inequality]
\label{thm:bennett}
Consider a finite sequence $\{\mat{X}_j\}$ of independent, random, self-adjoint matrices with dimension $n$, all of which have zero mean. Given an integer $k \leq n$, define
\[
 \sigma_k^2 = \lambda_k\left(\sum\nolimits_j \E(\mat{X}_j^2) \right). 
\]
Choose $\mat{V}_{+}\in \Isom{n-k+1}{n}$ to satisfy
\[
 \sigma_k^2 = \lambdamax{\sum\nolimits_j \mat{V}_{+}^\star\E(\mat{X}_j^2)\mat{V}_{+}}.
\]
Then, for all $t \geq 0,$
\begin{align}
\Prob{\lambda_k\left( \sum\nolimits_j \mat{X}_j \right) \geq t } & \leq (n-k+1) \cdot \exp\left\{-\frac{\sigma_k^2}{\randcon(\mat{V}_{+})^2} \cdot h\left(\frac{\randcon(\mat{V}_{+})t}{\sigma_k^2}\right) \right\} \tag{i} \label{eqn:bennettineq} \\
 & \leq (n-k+1) \cdot \exp\left\{ \frac{-t^2/2}{\sigma_k^2 + \randcon(\mat{V}_+)t/3} \right\} \tag{ii} \label{eqn:bernsteinineq} \\
 & \leq \begin{cases}[1.5]
        (n-k+1) \cdot \exp\left\{-\tfrac{3}{8}t^2/\sigma_k^2\right\}  & \text{ for } t \leq \sigma_k^2/\randcon(\mat{V}_+) \\
        (n-k+1) \cdot \exp\left\{-\tfrac{3}{8} t/\randcon(\mat{V}_+)\right\} & \text{ for } t \geq \sigma_k^2/\randcon(\mat{V}_+),  
        \end{cases} \tag{iii}  \label{eqn:splitbernsteinineqs} 
\end{align}
where the function $h(u) = (1+u)\log(1+u) - u$ for $u \geq 0.$ The function $\randcon$ satisfies \eqref{eqn:randcondef2} above.
\end{thm}

Results \eqref{eqn:bennettineq} and \eqref{eqn:bernsteinineq} are, respectively, matrix analogs of the classical Bennett and Bernstein inequalities. As in the scalar case, the Bennett inequality reflects a Poisson-type decay in the tails of the eigenvalues. The Bernstein inequality states that small deviations from the eigenvalues of the expected matrix are roughly normally distributed while larger deviations are subexponential. The split Bernstein inequalities \eqref{eqn:splitbernsteinineqs} make explicit the division between these two regimes.

As stated, Theorem \ref{thm:bennett} estimates the probability that the eigenvalues of a sum are large. Using the identity
\[
\lambda_k\left(\sum_j \mat{X}_j\right) = -\lambda_{n-k+1}\left(- \sum_j \mat{X}_j\right),
\]
Theorem \ref{thm:bennett} can be applied to estimate the probability that eigenvalues of a sum are small. 

To prove Theorem \ref{thm:bennett}, we use the following lemma (Lemma 6.7  in \cite{T10a}) to control the moment generating function of a random matrix with bounded maximum eigenvalue.
 
\begin{lemma}
Let $\mat{X}$ be a random self-adjoint matrix satisfying $\E\mat{X} = \mat{0}$ and $\lambdamax{\mat{X}} \leq 1$ almost surely. Then
\[ 
\E\e^{\theta\mat{X}} \preceq \exp((\e^{\theta} - \theta - 1) \cdot \E(\mat{X}^2)) \quad \text{for } \theta > 0. 
\]
\label{lemma:freedmanmgf}
\end{lemma}

\begin{proof}[Proof of Theorem \ref{thm:bennett}]
Using homogeneity, we assume without loss that $\randcon(\mat{V}_+)=1.$ This implies that $\lambdamax{\mat{X}_j} \leq 1$ almost surely for all the summands. By Lemma \ref{lemma:freedmanmgf},
\[
\E\e^{\theta \mat{X}_j} \preceq \exp\big( g(\theta) \cdot \E(\mat{X}_j^2) \big),
\]
with $g(\theta) = \e^\theta - \theta - 1.$ 

Theorem \ref{thm:eigtails}\eqref{eqn:uncompressedeigtails} then implies
\begin{align*}
\Prob{\lambda_k\left(\sum\nolimits_j \mat{X}_j \right) \geq t } & \leq \inf_{\theta > 0} \e^{-\theta t} \cdot \trexp{g(\theta) \sum\nolimits_j \mat{V}_{+}^\star\E (\mat{X}_j^2)\mat{V}_{+}} \\
 & \leq (n-k+1) \cdot \inf_{\theta >0} \e^{-\theta t} \cdot \lambdamax{\exp\left\{g(\theta)\sum\nolimits_j \mat{V}_{+}^\star\E(\mat{X}_j^2)\mat{V}_{+} \right\}} \\
& = (n-k+1) \cdot \inf_{\theta > 0} e^{-\theta t} \cdot \exp\left\{ g(\theta) \cdot \lambdamax{\sum\nolimits_j \mat{V}_{+}^\star\E(\mat{X}_j^2)\mat{V}_{+} } \right\}.
\end{align*}
The maximum eigenvalue in this expression equals $\sigma_k^2$, thus
\[ 
\Prob{\lambda_k\left(\sum\nolimits_j \mat{X}_j \right) \geq t} \leq (n-k+1) \cdot \inf_{\theta >0 } \e^{g(\theta) \sigma_k^2 -\theta t}.
\]
The Bennett inequality \eqref{eqn:bennettineq} follows by substituting $\theta = \log(1 + t/\sigma_k^2)$ into the right-hand side and simplifying.

The Bernstein inequality \eqref{eqn:bernsteinineq} is a consequence of \eqref{eqn:bennettineq} and the fact that 
\[
 h(u) \geq \frac{u^2/2}{1 + u/3} \quad \text{ for } u \geq 0,
\]
which can be established by comparing derivatives.

The subgaussian and subexponential portions of the split Bernstein inequalities \eqref{eqn:splitbernsteinineqs} are verified through algebraic comparisons on the relevant intervals.
\end{proof}

Occasionally, as in the application in section \ref{sec:covarianceest} to the problem of covariance matrix estimation, one desires a Bernstein-type tail bound that applies to summands that do not have bounded maximum eigenvalues. In this case, if the moments of the summands satisfy sufficiently strong growth restrictions, one can extend classical scalar arguments to obtain results such as the following Bernstein bound for subexponential matrices. 

\begin{thm}[Eigenvalue Bernstein Inequality for Subexponential Matrices]
Consider a finite sequence $\{\mat{X}_j\}$ of independent, random, self-adjoint matrices with dimension $n$, all of which satisfy the subexponential moment growth condition
\[
\E (\mat{X}_j^m) \preceq \frac{m!}{2} B^{m-2} \mat{\Sigma}_j^2 \quad \text{ for } m=2,3,4,\ldots,
\]
where $B$ is a positive constant and $\mat{\Sigma}_j^2$ are positive-semidefinite matrices. Given an integer $k \leq n$, set
\[
\mu_k = \lambda_k \left( \sum\nolimits_j \E \mat{X}_j \right).
\]
Choose $\mat{V}_+ \in \Isom{n-k+1}{n}$ that satisfies
\[
\mu_k = \lambdamax{ \sum\nolimits_j \mat{V}_+^\star (\E \mat{X}_j) \mat{V}_+ },
\]
and define
\[
 \quad \sigma_k^2 = \lambdamax{ \sum\nolimits_j \mat{V}_+^\star \mat{\Sigma}_j^2 \mat{V}_+ }.
\] 
Then, for any $t \geq 0,$
\begin{align}
\Prob{\lambda_k \left( \sum\nolimits_j \mat{X}_j \right) \geq \mu_k + t } & \leq (n-k+1) \cdot \exp\left\{ - \frac{t^2/2}{\sigma_k^2 + B t}\right\} \tag{i} \label{eqn:subexponentialbernstein} \\
 & \leq \begin{cases}[1.5]
         (n-k+1) \cdot \exp\left\{ -\tfrac{1}{4} t^2/\sigma_k^2\right\} & \text{ for } t \leq \sigma_k^2/B \\
	 (n-k+1) \cdot \exp\left\{ -\tfrac{1}{4} t/B\right\} & \text{ for } t \geq \sigma_k^2/B.
        \end{cases} \tag{ii} \label{eqn:splitsubexponentialbernstein}
\end{align}
\label{thm:subexponentialbernstein}
\end{thm}
 
This result is an extension of \cite[Theorem 6.2]{T10a}, which, in turn, generalizes a classical scalar argument \cite{DG99}.

As with the other matrix inequalities, Theorem \ref{thm:subexponentialbernstein} follows from an application of Theorem \ref{thm:eigtails} and appropriate semidefinite bounds on the moment generating functions of the summands. Thus, the key to the proof lies in exploiting the moment growth conditions of the summands to majorize their moment generating functions. The following lemma, a trivial extension of Lemma 6.8 in \cite{T10a}, provides what we need.

\begin{lemma}
Let $\mat{X}$ be a random self-adjoint matrix satisfying the subexponential moment growth conditions
\[
\E (\mat{X}^m) \preceq \frac{m!}{2} \mat{\Sigma}^2 \quad \text{for } m=2,3,4,\ldots.
\]
Then, for any $\theta$ in $[0,1),$
\[
 \E \exp(\theta \mat{X}) \preceq \exp\left( \theta \E \mat{X} + \frac{\theta^2}{2(1 - \theta)} \mat{\Sigma}^2 \right).
\]
\label{lemma:growthbernsteinmgf}
\end{lemma}

\begin{proof}[Proof of Theorem \ref{thm:subexponentialbernstein}]

We note that $\mat{X}_j$ satisfies the growth condition
\[
 \E (\mat{X}_j^m) \preceq \frac{m!}{2} B^{m-2} \mat{\Sigma}_j^2 \quad \text{for } m \geq 2 
\]
if and only if the scaled matrix $\mat{X}_j/B$ satisfies
\[
 \E \left(\frac{\mat{X}_j}{B}\right)^m \preceq \frac{m!}{2} \cdot \frac{\mat{\Sigma}_j^2}{B^2} \quad \text{for } m \geq 2.
\]
Thus, by rescaling, it suffices to consider the case $B = 1.$ We now do so.

By Lemma \ref{lemma:growthbernsteinmgf}, the moment generating functions of the summands satisfy  
\[
 \E \exp(\theta \mat{X}_j) \preceq \exp\left(\theta \E \mat{X}_j + g(\theta) \mat{\Sigma}_j^2 \right),
\]
where $g(\theta) = \theta^2/(2 -2 \theta).$ Now we apply Theorem \ref{thm:eigtails}\eqref{eqn:uncompressedeigtails}:
\begin{align*}
 \Prob{\lambda_k\left(\sum\nolimits_j \mat{X}_j\right) \geq \mu_k + t} & \leq \inf_{\theta \in [0,1)} \e^{-\theta(\mu_k + t)} \cdot \trexp{ \theta \sum\nolimits_j \mat{V}_+^\star (\E\mat{X}_j) \mat{V}_+ + g(\theta) \sum\nolimits_j \mat{V}_+^\star \mat{\Sigma}_j^2 \mat{V}_+ } \\
& \leq \inf_{\theta \in [0,1)} (n -k + 1)\cdot \exp\Big\{-\theta(\mu_k + t) + \theta \cdot \lambdamax{\sum\nolimits_j \mat{V}_+^\star (\E\mat{X}_j) \mat{V}_+}   \\
&  \qquad {} + g(\theta) \cdot \lambdamax{\sum\nolimits_j \mat{V}_+^\star \mat{\Sigma}_j^2 \mat{V}_+} \Big\} \\
 & = \inf_{\theta \in [0, 1)} (n-k+1)\cdot \exp\left( -\theta t + g(\theta) \sigma_k^2 \right).
\end{align*}
To achieve the final simplification, we identified $\mu_k$ and $\sigma_k^2.$
Now, select $\theta = t/(t + \sigma_k^2).$ Then simplication gives the Bernstein inequality \eqref{eqn:subexponentialbernstein}. 

Algebraic comparisons on the relevant intervals yield the split Bernstein inequalities \eqref{eqn:splitsubexponentialbernstein}.
\end{proof}

\section{An application to column subsampling}
\label{sec:colsubsampling}

As an application of our Chernoff bounds, we examine how sampling columns from a matrix with orthonormal rows affects the spectrum. This question has applications in numerical linear algebra and compressed sensing. The special cases of the maximum and minimum eigenvalues have been studied in the literature \cite{Tropp08, RV07}. The limiting spectral distributions of matrices formed by sampling columns from similarly structured matrices have also been studied: the results of \cite{GH09} apply to matrices formed by sampling columns from any fixed orthogonal matrix, and \cite{F10} studies matrices formed by sampling columns and rows from the discrete Fourier transform matrix.  We mention in particular \cite{RU99}, the main result of which provides a uniform bound on the tails of all singular values of the sampled matrix. The theorem proven in this section provides bounds which reflect the differences in the tails of the individual singular values, and thus can be viewed as an elaboration of the result in \cite{RU99}.

Let $\mat{U}$ be an $n \times r$ matrix with orthonormal rows. We model the sampling operation using a random diagonal matrix $\mat{D}$ whose entries are independent $\text{Bern}(p)$ random variables. Then the random matrix
\begin{equation}
\widehat{\mat{U}} = \mat{U}\mat{D}
\label{eqn:colsubsampling}
\end{equation}
can be interpreted as a random column submatrix of $\mat{U}$ with an average of $p r$ nonzero columns. Our goal is to study the behavior of the spectrum of $\widehat{\mat{U}}.$

Recall that the $j$th column of $\mat{U}$ is written $\vec{u}_j.$ Consider the following coherence-like quantity associated with $\mat{U}:$
\begin{equation}
 \tau_k  =  \min_{\mat{V} \in \Isom{k}{n}} \max\nolimits_j \norm{\mat{V}^\star \vec{u}_j}^2 \quad \text{for } k=1,\ldots,n. 
 \label{eqn:coherencelikequantity}
\end{equation}
There does not seem to be a simple expression for $\tau_k.$ However, by choosing $\mat{V}^\star$ to be the restriction to an appropriate $k$-dimensional coordinate subspace, we see that $\tau_k$ always satisfies
\[
 \tau_k \leq \min_{|I|\leq k} \max\nolimits_j \sum_{i \in I} u_{ij}^2. 
\]

The following theorem shows that the behavior of $\s_k(\widehat{\mat{U}}),$ the $k$th singular value of $\widehat{\mat{U}},$ can be explained in terms of $\tau_k.$
\begin{thm}[Column Subsampling of Matrices with Orthonormal Rows]
\label{thm:colsampling}
Let $\mat{U}$ be an $n \times r$ matrix with orthonormal rows, and let $p$ be a sampling probability. Define the sampled matrix $\widehat{\mat{U}}$ according to \eqref{eqn:colsubsampling}, and the numbers $\{\tau_k\}$ according to \eqref{eqn:coherencelikequantity}.
Then, for each $k=1,\ldots,n,$
\begin{align*}
\Prob{\s_k(\widehat{\mat{U}}) \geq \sqrt{(1 + \delta) p}} & \leq (n-k+1) \cdot \left[ \frac{\e^\delta}{(1+\delta)^{1+\delta}} \right]^{p/\tau_{n-k+1}} & \text{for $\delta > 0$} \\
\Prob{\s_k(\widehat{\mat{U}}) \leq \sqrt{(1- \delta)p} } & \leq k \cdot \left[ \frac{\e^{-\delta}}{(1-\delta)^{1-\delta}} \right]^{p/\tau_k } & \text{for $\delta \in [0,1).$} 
\end{align*}

\end{thm}

\begin{proof}
Observe, using \eqref{eqn:colsubsampling}, that
\[
 \s_k(\widehat{\mat{U}})^2 = \lambda_k(\mat{U} \mat{D}^2 \mat{U}^\star ) = \lambda_k \left( \sum_j d_j \vec{u}_j \vec{u}_j^\star \right),
\]
where $\vec{u}_j$ is the $j$th column of $\mat{U}$ and $d_j \sim \text{Bern}(p).$
Compute 
\[
 \mu_k = \lambda_k\left(\sum\nolimits_j \E d_j \vec{u}_j \vec{u}_j^\star \right) = p \cdot \lambda_k(\mat{U} \mat{U}^\star) = p \cdot \lambda_k(\mathbf{I}) = p.
\]
It follows that, for \emph{any} $\mat{V} \in \Isom{n-k+1}{n},$
\[
\lambdamax{\sum\nolimits_j \mat{V}^\star (\E d_j \vec{u}_j \vec{u}_j^\star) \mat{V} } = p \cdot \lambdamax{\mat{V}^\star \mat{V}} = p = \mu_k,
\]
so the choice of $\mat{V}_+ \in \Isom{n-k+1}{n}$ is arbitrary. Similarly, the choice of $\mat{V}_- \in \Isom{k}{n}$ is arbitrary. We select $\mat{V}_+$ to be an isometric embedding that achieves $\tau_{n-k+1}$ and $\mat{V}_-$ to be an isometric embedding that achieves $\tau_k$. Accordingly,
\begin{align*}
 \randcon(\mat{V}_+) & = \max\nolimits_j \|\mat{V}_+^* \vec{u}_j \vec{u}_j^* \mat{V}_+\| = \max\nolimits_j \|\mat{V}_+^* \vec{u}_j\|^2 = \tau_{n-k+1}, \quad \text{and} \\
 \randcon(\mat{V}_-) & = \max\nolimits_j \|\mat{V}_-^* \vec{u}_j \vec{u}_j^* \mat{V}_-\| = \max\nolimits_j \|\mat{V}_-^* \vec{u}_j\|^2 = \tau_{k}.
\end{align*}
Theorem \ref{thm:chernoff} delivers the upper bound
\begin{align*}
\Prob{\s_k(\hat{\mat{U}}) \geq \sqrt{(1 + \delta) p}} & = \Prob{\lambda_k\left(\sum_j d_j \vec{u}_j \vec{u}_j^\star \right) \geq (1 + \delta) p} 
\leq (n-k+1) \cdot \left[ \frac{\e^\delta}{(1+\delta)^{1+\delta}} \right]^{p/\tau_{n-k+1}}
\intertext{ for $\delta > 0$ and the lower bound }
\Prob{\s_k(\hat{\mat{U}}) \leq \sqrt{(1- \delta)p} } & = \Prob{\lambda_k\left(\sum_j d_j \vec{u}_j \vec{u}_j^\star \right) \leq (1 - \delta) p}
 \leq k \cdot \left[ \frac{\e^{-\delta}}{(1-\delta)^{1-\delta}} \right]^{p/\tau_k }
\end{align*}
for $\delta \in [0,1).$
\end{proof}

\begin{figure}[t!]
\includegraphics{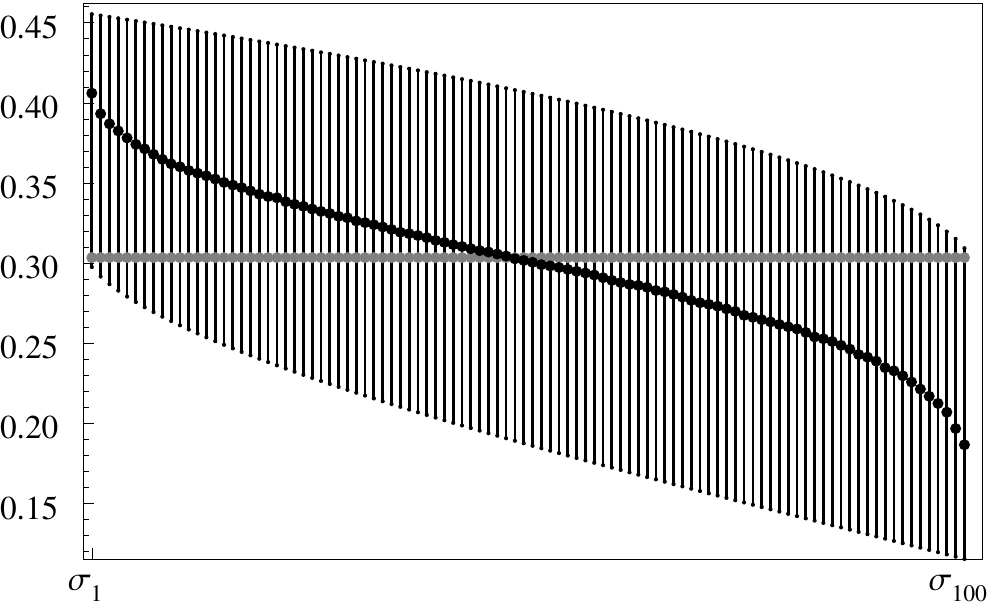}
\caption{[Spectrum of a random submatrix] The matrix $\mat{U}$ is a $10^2 \times 10^4$ submatrix of the unitary DFT matrix with dimension $10^4,$ and the sampling probability $p = 10^{-4} \log(10^4).$ The $k$th vertical bar, calculated using Theorem \ref{thm:colsampling}, describes an interval containing the median value of the $k$th singular value of the sampled matrix $\widehat{\mat{U}}$. The black circles denote the empirical medians of the singular values of $\widehat{\mat{U}}$, calculated from 500 trials. The gray circles represent the singular values of $\E \widehat{\mat{U}}.$}  
\label{fig:nlogn}
\end{figure}

To illustrate the discriminatory power of these bounds, let $\mat{U}$ be an $n \times n^2$ matrix consisting of $n$ rows of the $n^2 \times n^2$ Fourier matrix and choose $p = (\log n)/n$ so that, on average, sampling reduces the aspect ratio from $n$ to $\log n.$  For $n=100,$ we determine upper and lower bounds for the median value of $\s_k(\widehat{\mat{U}})$ by numerically finding the value of $\delta$ where the probability bounds in Theorem \ref{thm:colsampling} equal $1/2.$ Figure \ref{fig:nlogn} plots the empirical median value along with the computed interval. We see that these ranges reflect the behavior of the singular values more faithfully than the simple estimates $\s_k(\mathbb{E} \widehat{\mat{U}}) = p.$

\section{Covariance Estimation}
\label{sec:covarianceest}

We conclude with an extended example that illustrates how this circle of ideas allows one to answer interesting statistical questions. Specifically, we investigate the convergence of the individual eigenvalues of sample covariance matrices, with errors measured in \emph{relative} precision.

Covariance estimation is a basic and ubiquitious problem that arises in signal processing, graphical modeling, machine learning, and genomics, among other areas. Let $\{\vec{\eta}_j\}_{j=1}^n \subset \R^p$ be i.i.d.~samples drawn from some distribution with zero mean and covariance matrix $\mat{C}.$ Define the sample covariance matrix 
\[
 \widehat{\mat{C}}_n = \frac{1}{n} \sum_{j=1}^n \vec{\eta}_j\vec{\eta}_j^\star.
\]
An important challenge is to determine how many samples are needed to ensure that the empirical covariance estimator has a fixed relative accuracy in the spectral norm. That is, given a fixed $\varepsilon,$ how large must $n$ be so that 
\begin{equation}
\label{eqn:relacccovar}
 \snorm{\widehat{\mat{C}}_n - \mat{C}} \leq \varepsilon \snorm{\mat{C}}?
\end{equation}

This estimation problem has been studied extensively. It is now known that for distributions with a finite second moment, $\Omega(p \log p)$ samples suffice \cite{RU99}, and for log-concave distributions, $\Omega(p)$ samples suffice \cite{ALPT10b}. More broadly, Vershynin \cite{V10} conjectures that, for distributions with finite fourth moment, $\Omega(p)$ samples suffice; he establishes this result to within iterated log factors. In \cite{SrivastavaVershynin11}, Srivastava and Vershynin establish that $\Omega(p)$ samples suffice for distributions which have finite $2+\varepsilon$ moments, for some $\varepsilon >0,$ and satisfy an additional regularity condition.

Inequality \eqref{eqn:relacccovar} ensures that the difference between the $k$th eigenvalues of $\widehat{\mat{C}}_n$ and $\mat{C}$ is small, but it requires $\mathrm{O}(p)$ measurements to obtain estimates of even a few of the eigenvalues. Specifically, letting $\kappa_\ell = \lambda_1(\mat{C})/\lambda_\ell(\mat{C}),$ we see that $O(\varepsilon^{-2}\kappa_\ell^2 p)$ measurements are required to obtain relative-error estimates of the dominant $\ell$ eigenvalues of $\mat{C}$ using the results of \cite{ALPT10b,V10,SrivastavaVershynin11}. However, it is reasonable to expect that when the spectrum of $\mat{C}$ exhibits decay and $\ell \ll p,$ much fewer than $\mathrm{O}(p)$ measurements should suffice for relative-error recovery of the dominant $\ell$ eigenvalues. 

In this section, we derive a relative approximation bound for each eigenvalue of $\mat{C}$ that allows us to confirm this intuition. For simplicity we assume the samples are drawn from a $\mathcal{N}(\vec{0}, \mat{C})$ distribution where $\mat{C}$ is full-rank, but the arguments can be extended to cover other distributions.

\begin{thm}
\label{thm:covarest}
Assume that $\mat{C} \in \samats{p}$ is positive definite. Let $\{\vec{\eta}_j\}_{j=1}^n \subset \R^p$ be i.i.d.~samples drawn from a $\mathcal{N}(\vec{0}, \mat{C})$ distribution. Define 
\[
\widehat{\mat{C}}_n = \frac{1}{n} \sum\nolimits_{j=1}^n \vec{\eta}_j\vec{\eta}_j^\star.
\]
Write $\lambda_k$ for the $k$th eigenvalue of $\mat{C}$, and write $\hat{\lambda}_k$ for the $k$th eigenvalue of $\widehat{\mat{C}}_n.$ Then for $k=1,\ldots,p,$
\begin{align*}
\Prob{\hat{\lambda}_k \geq \lambda_k + t} & \leq  (p-k+1) \cdot \exp\left( \displaystyle \frac{-\mathrm{c} nt^2}{\lambda_k \sum_{i=k}^p \lambda_i} \right) \quad \text{for } t \leq 4 n \lambda_k, \\
\intertext{ and }
\Prob{\hat{\lambda}_k \leq \lambda_k - t} & \leq  k \cdot \exp\left( \displaystyle \frac{-\mathrm{c}nt^2}{ \lambda_1 \sum_{i=1}^k \lambda_i } \right) \quad \text{for } t \leq 4n\lambda_1,
\end{align*}
where the constant $\mathrm{c}$ is at least $1/32.$
\end{thm}

The following corollary provides an answer to our question about relative error estimates.
\begin{cor}
\label{cor:relerrcovarest}
Let $\lambda_k$ and $\hat{\lambda}_k$ be as in Theorem \ref{thm:covarest}. Then 
\begin{align*}
\Prob{\hat{\lambda}_k \geq (1 + \varepsilon) \lambda_k} & \leq 
(p-k+1) \cdot \exp\left(\frac{- \mathrm{c} n \varepsilon^2 }{\sum\nolimits_{i=k}^p \frac{\lambda_i}{\lambda_k}} \right) \quad \text{for } \varepsilon \leq 4 n, \\
\intertext{ and }
\Prob{\hat{\lambda}_k  \leq (1 - \varepsilon) \lambda_k} & \leq k \cdot \exp\left( \frac{- \mathrm{c} n \varepsilon^2}{\frac{\lambda_1}{\lambda_k}\sum\nolimits_{i=1}^k \frac{\lambda_i}{\lambda_k}} \right) \quad \text{for } \varepsilon \in (0,1],
\end{align*}
where the constant $\mathrm{c}$ is at least $1/32.$
\end{cor}

The first bound in Corollary \ref{cor:relerrcovarest} tells us how many samples are needed to ensure that $\hat{\lambda}_k$ does not overestimate $\lambda_k.$ Likewise, the second bound tells us how many samples ensure that $\hat{\lambda}_k$ does not underestimate $\lambda_k.$

Corollary \ref{cor:relerrcovarest} suggests that the relationship of $\hat{\lambda}_k$ to $\lambda_k$ is determined by the spectrum of $\mat{C}$ in the following manner. When the eigenvalues below $\lambda_k$ are small compared with $\lambda_k$, the quantity
\[
 \sum_{i=k}^p \lambda_i/\lambda_k
\]
is small, and so $\hat{\lambda}_k$ is not likely to overestimate $\lambda_k$. Similarly, when the eigenvalues above $\lambda_k$ are comparable with $\lambda_k$, the quantity
\[
\frac{\lambda_1}{\lambda_k}\sum_{i=1}^k \lambda_i/\lambda_k 
\]
is small, and so $\hat{\lambda}_k$ is not likely to underestimate $\lambda_k$.

We now have everything needed to establish Theorem~\ref{thm:examplecovarest}.


\begin{proof}[Proof of Theorem \ref{thm:examplecovarest} from Corollary \ref{cor:relerrcovarest}]
From Corollary \ref{cor:relerrcovarest}, we see that 
\[
\Prob{\hat{\lambda}_k \leq (1-\varepsilon) \lambda_k } \leq p^{-\beta} \quad \text{when}\quad n \geq 32 \varepsilon^{-2} \left( \frac{\lambda_1}{\lambda_k} \sum\nolimits_{i\leq k} \frac{\lambda_i}{\lambda_k} \right) (\log k + \beta \log p).
\]
Recall that $\kappa_k = \lambda_1(\mat{C})/\lambda_k(\mat{C}).$ Clearly, taking $n = \Omega(\varepsilon^{-2} \kappa_\ell^2 \ell \log p)$ samples ensures that, with high probability, each of the top $\ell$ eigenvalues of the sample covariance matrix satisfies $\hat{\lambda}_k > (1-\varepsilon) \lambda_k.$

Likewise, 
\[
\Prob{\hat{\lambda}_k \geq (1+\varepsilon) \lambda_k } \leq p^{-\beta} \quad\text{when}\quad n \geq 32 \varepsilon^{-2} \left(\sum_{i \geq k} \frac{\lambda_i }{\lambda_k} \right)(\log (p-k+1) + \beta \log p).
\]
Assuming the stated decay condition, that
\[
\sum\nolimits_{i > \ell} \lambda_i = \mathrm{O}(\lambda_1),
\]
we see that taking $n = \Omega(\varepsilon^{-2} (\ell + \kappa_\ell) \log p)$ samples ensures that, with high probability, each of the top $\ell$ eigenvalues of the sample covariance matrix satisfies $\hat{\lambda}_k < (1+\varepsilon) \lambda_k.$

Combining these two results, we conclude that $n = \Omega( \varepsilon^{-2} \kappa_\ell^2 \ell \log p)$ ensures that the top $\ell$ eigenvalues of $\mat{C}$ are estimated to within relative precision $1 \pm \varepsilon.$
\end{proof}

\begin{remark}
The results in Theorem \ref{thm:covarest} and Corollary \ref{cor:relerrcovarest} also apply when $\mat{C}$ is rank-deficient: simply replace each occurence of the dimension $p$ in the bounds with $\rank(\mat{C}).$ 
\end{remark}

\subsection{Proof of Theorem \ref{thm:covarest}}
We now prove Theorem \ref{thm:covarest}. This result requires supporting lemmas; we defer their proofs until after a discussion of extensions to Theorem \ref{thm:covarest}.

We study the error $|\lambda_k(\widehat{\mat{C}}_n) - \lambda_k(\mat{C})|.$ To apply the methods developed in this paper, we pass to a question about the eigenvalues of a difference of two matrices. The first lemma accomplishes this goal by compressing both the population covariance matrix and the sample covariance matrix to a fixed invariant subspace of the population covariance matrix.

\begin{lemma}
\label{lemma:splittail}
Let $\mat{X}$ be a random self-adjoint matrix with dimension $p,$ and let $\mat{A}$ be a fixed self-adjoint matrix with dimension $p$. Choose $\mat{W}_+ \in \Isom{p-k+1}{p}$ and $\mat{W}_- \in \Isom{k}{p}$ for which
\[
  \lambda_k(\mat{A}) = \lambdamax{\mat{W}_+^*\mat{A}\mat{W}_+} = \lambdamin{\mat{W}_-^*\mat{A}\mat{W}_-}.
\]
Then, for all $t >0,$
\begin{align}
 \Prob{\lambda_k(\mat{X}) \geq \lambda_k(\mat{A}) + t} & \leq \Prob{\lambdamax{\mat{W}_+^\star \mat{X} \mat{W}_+} \geq \lambda_k(\mat{A}) + t } \label{eqn:covardeterministicupperbnd}\\
\intertext{ and }
 \Prob{\lambda_k(\mat{X}) \leq \lambda_k(\mat{A}) -t } &\leq \Prob{\lambdamax{\mat{W}_-^\star(- \mat{X}) \mat{W}_-} \geq -\lambda_k(\mat{A}) + t}.
 \label{eqn:covardeterministiclowerbnd}
\end{align}
\end{lemma}

We apply this result with $\mat{A} = \mat{C}$ and $\mat{X} = \widehat{\mat{C}}_n.$ Because $\widehat{\mat{C}}_n$ is unbounded, we apply Theorem \ref{thm:subexponentialbernstein} to handle the estimates in \eqref{eqn:covardeterministicupperbnd} and \eqref{eqn:covardeterministiclowerbnd}. To use this theorem, we need the following moment growth estimate for rank-one Wishart matrices.

\begin{lemma}
\label{lemma:momentbound}
Let $\vec{\xi} \sim \mathcal{N}(\vec{0},\mat{G}).$ Then for any integer $m \geq 2,$
$$
\E\left(\vec{\xi}\vec{\xi}^\star \right)^m \preceq 2^m m! (\tr \mat{G})^{m-1}\cdot \mat{G}.
$$
\end{lemma}

With these preliminaries addressed, we prove Theorem \ref{thm:covarest}.

\begin{proof}[Proof of upper estimate]
First we consider the probability that $\hat{\lambda}_k$ overestimates $\lambda_k$. Let $\mat{W}_+ \in \Isom{p-k+1}{p}$ satisfy
\[
 \lambda_k(\mat{C}) = \lambdamax{\mat{W}_+^\star \mat{C} \mat{W}_+}.
\]
Then Lemma \ref{lemma:splittail} implies
\begin{align}
 \Prob{\lambda_k(\widehat{\mat{C}}_n) \geq \lambda_k(\mat{C}) + t} & \leq \Prob{\lambdamax{\mat{W}_+^\star \widehat{\mat{C}}_n \mat{W}_+} \geq \lambda_k(\mat{C}) + t} \notag \\
& = \Prob{\lambdamax{\sum\nolimits_j \mat{W}_+^\star (\vec{\eta}_j \vec{\eta}_j^\star) \mat{W}_+} \geq n \lambda_k(\mat{C}) + n t}.
\label{eqn:upperestimateprob}
\end{align}
The factor $n$ comes from the normalization of the sample covariance matrix.

The covariance matrix of $\vec{\eta}_j$ is $\mat{C},$ so that of $\mat{W}_+^\star \vec{\eta}_j$ is $\mat{W}_+^\star\mat{C}\mat{W}_+.$ Apply Lemma \ref{lemma:momentbound} to verify that $\mat{W}_+^\star \vec{\eta}_j \vec{\eta}_j \mat{W}_+$ satisfies the subexponential moment growth bound required by Theorem \ref{thm:subexponentialbernstein} with 
\[
 B = 2 \tr(\mat{W}_+^*\mat{C}\mat{W}_+) \quad\text{ and }\quad \mat{\Sigma}_j^2 = 8\tr(\mat{W}_+^*\mat{C}\mat{W}_+)\cdot \mat{W}_+^*\mat{C}\mat{W}_+.
\]
In fact, $\mat{W}_+^*\mat{C}\mat{W}_+$ is the compression of $\mat{C}$ to the invariant subspace corresponding with its bottom $p-k+1$ eigenvalues, so 
\[
B = 2 \sum\nolimits_{i=k}^p \lambda_i(\mat{C}) \quad\text{and}\quad \lambdamax{\mat{\Sigma}_j^2} = 8 \lambda_k(\mat{C}) \sum\nolimits_{i=k}^p \lambda_i(\mat{C}).
\]
We are concerned with the maximum eigenvalue of the sum in \eqref{eqn:upperestimateprob}, so we take $\mat{V}_+ = \mathbf{I}$ in the statement of Theorem \ref{thm:subexponentialbernstein} to find that
\begin{align*}
 \sigma_1^2 & = \lambdamax{\sum\nolimits_j \mat{\Sigma}_j^2 } = n \lambdamax{\mat{\Sigma}_1^2} = 8 n \lambda_k(\mat{C}) \sum\nolimits_{i=k}^p \lambda_i(\mat{C}) \quad \text{and}
\\ \mu_1 & = \lambdamax{\sum_j \mat{W}_+^\star \E(\vec{\eta}_j\vec{\eta}_j^\star) \mat{W}_+ } = n\lambdamax{\mat{W}_+^\star \mat{C} \mat{W}_+} = n \lambda_k(\mat{C}).
\end{align*}
It follows from the subgaussian branch of the split Bernstein inequality of Theorem \ref{thm:subexponentialbernstein} that 
\[
\Prob{\lambdamax{\sum\nolimits_j \mat{W}_+^\star (\vec{\eta}_j \vec{\eta}_j^\star) \mat{W}_+} \geq n \lambda_k(\mat{C}) + n t} \leq 
 (p-k+1) \cdot \exp\left( \displaystyle \frac{-nt^2}{32 \lambda_k(\mat{C}) \sum_{i=k}^p \lambda_i(\mat{C})} \right)
\]
when $t \leq 4 n \lambda_k(\mat{C}).$ This provides the desired bound on the probability that $\lambda_k(\widehat{\mat{C}}_n)$ overestimates $\lambda_k(\mat{C}).$
\end{proof}

\begin{proof}[Proof of lower estimate]
Now we consider the probability that $\hat{\lambda}_k$ underestimates $\lambda_k.$ The proof proceeds similarly to the proof of the upper estimate. Let $\mat{W}_- \in \Isom{k}{p}$ satisfy
\[
\lambda_k(\mat{C}) = \lambdamin{\mat{W}_-^\star \mat{C} \mat{W}_-}.
\]
Then Lemma \ref{lemma:splittail} implies
\begin{align}
\Prob{\lambda_k(\widehat{\mat{C}}_n) \leq \lambda_k(\mat{C}) - t} & \leq \Prob{\lambdamax{\mat{W}_-^\star (-\widehat{\mat{C}}_n ) \mat{W}_-} \geq -n \lambda_k(\mat{C}) + nt} \notag \\
 & = \Prob{\lambdamax{\sum\nolimits_j \mat{W}_-^\star (-\mat{\eta}_j \mat{\eta}_j^\star) \mat{W}_-} \geq -n \lambda_k(\mat{C}) + nt}
\label{eqn:lowerestimateprob}
\end{align}
The factor $n$ comes from the normalization of the sample covariance matrix.

The covariance matrix of $\vec{\eta}_j$ is $\mat{C},$ so that of $\mat{W}_-^\star \vec{\eta}_j$ is $\mat{W}_-^\star \mat{C} \mat{W}_-.$ Apply Lemma \ref{lemma:momentbound} to verify that for any integer $m \geq 2,$
\[
\E(\mat{W}_-^\star(- \vec{\eta}_j \vec{\eta}_j^\star) \mat{W}_-)^m \preceq \E(\mat{W}_-^\star \vec{\eta}_j \vec{\eta}_j^\star \mat{W}_-)^m \preceq 2^m m!\tr(\mat{W}_-^\star \mat{C} \mat{W}_-)^{m-1} \cdot \mat{W}_-^\star \mat{C} \mat{W}_-. 
\]
Thus, $\mat{W}_-^\star (-\vec{\eta}_j \vec{\eta}_j^\star) \mat{W}_-$ satisfies the subexponential moment growth bound required by Theorem \ref{thm:subexponentialbernstein} with
\[
B = 2 \tr(\mat{W}_-^\star \mat{C} \mat{W}_-) \quad \text{and} \quad \mat{\Sigma}_j^2 = 8 \tr(\mat{W}_-^\star \mat{C} \mat{W}_-) \cdot \mat{W}_-^\star \mat{C} \mat{W}_-.
\]
In fact, $\mat{W}_-^\star \mat{C} \mat{W}_-$ is the compression of $\mat{C}$ to the invariant subspace corresponding with its top $k$ eigenvalues, so
\[
B = 2 \sum\nolimits_{i=1}^k \lambda_i(\mat{C}) \quad \text{and} \quad \lambdamax{\mat{\Sigma}_j^2} = 8\lambda_1(\mat{C}) \sum\nolimits_{i=1}^k \lambda_i(\mat{C}).
\]
We are concerned with the maximum eigenvalue of the sum in \eqref{eqn:lowerestimateprob}, so we take $\mat{V}_+ = \mathbf{I}$ in the statement of Theorem \ref{thm:subexponentialbernstein} to find that 
\begin{align*}
\sigma_1^2 & = \lambdamax{\sum\nolimits_j \mat{\Sigma}_j^2} = n \lambdamax{\mat{\Sigma}_1^2} = 8n \lambda_1(\mat{C}) \sum\nolimits_{i=1}^k \lambda_i(\mat{C}) \quad \text{and} \\
\mu_1 & = \lambdamax{\sum\nolimits_j \mat{W}_-^\star \E(-\vec{\eta}_j \vec{\eta}_j^\star) \mat{W}_-} = n \lambdamax{\mat{W}_-^\star(-\mat{C})\mat{W}_-} = -n\lambda_k(\mat{C}).
\end{align*}
It follows from the subgaussian branch of the split Bernstein inequality of Theorem \ref{thm:subexponentialbernstein} that
\[
\Prob{\lambdamax{\sum\nolimits_j \mat{W}_-^\star (-\vec{\eta}_j \vec{\eta}_j^\star) \mat{W}_-} \geq -n\lambda_k(\mat{C}) + nt} \leq k \cdot \exp\left(\displaystyle \frac{-nt^2}{32 \lambda_1(\mat{C}) \sum\nolimits_{i=1}^k \lambda_i(\mat{C})} \right)
\]
when $t \leq 4n\lambda_1(\mat{C}).$ This provides the desired bound on the probability that $\lambda_k(\widehat{\mat{C}}_n)$ underestimates $\lambda_k(\mat{C}).$
\end{proof}

\subsection{Extensions of Theorem~\ref{thm:covarest}}
Results analogous to Theorem~\ref{thm:covarest} can be established for other distributions. If the distribution is bounded, the possibility that $\hat{\lambda}_k$ deviates above or below $\lambda_k$ can be controlled using the Bernstein inequality of Theorem \ref{thm:bennett}. If the distribution is unbounded but has matrix moments that satisfy a sufficiently nice growth condition, the probability that $\hat{\lambda}_k$ deviates below $\lambda_k$ as well as the probability that it deviates above $\lambda_k$ can be bounded using a Bernstein inequality analogous to that in Theorem \ref{thm:subexponentialbernstein}.

Theorem~\ref{thm:covarest} controls the error in the $k$th sample eigenvalue in terms of all the eigenvalues of the covariance matrix, so it is most useful when the eigenvalues of the covariance matrix satisfy decay conditions such as those given in the statement of Theorem~\ref{thm:examplecovarest}. If such conditions are not satisfied, the results of \cite{ALPT10b} on the convergence of empirical covariance matrices of isotropic log-concave random vectors lead to tighter bounds on the probabilities that $\hat{\lambda}_k$ overestimates or underestimates $\lambda_k.$ 

To see the relevance of the results in \cite{ALPT10b}, first observe the following consequence of the subadditivity of the maximum eigenvalue mapping:
\begin{align*}
\lambdamax{\mat{W}_+^\star(\mat{X} - \mat{A}) \mat{W}_+} & \geq \lambdamax{\mat{W}_+^\star \mat{X} \mat{W}_+} - \lambdamax{\mat{W}_+^\star \mat{A} \mat{W}_+} \\
 &= \lambdamax{\mat{W}_+^\star \mat{X} \mat{W}_+} - \lambda_k(\mat{A}).
\end{align*}
In conjunction with \eqref{eqn:covardeterministicupperbnd}, this gives us the following control on the probability that $\lambda_k(\mat{X})$ overestimates $\lambda_k(\mat{A}):$ 
\[
 \Prob{\lambda_k(\mat{X}) \geq \lambda_k(\mat{A}) + t} \leq \Prob{\lambdamax{\mat{W}_+^\star (\mat{X} - \mat{A}) \mat{W}_+} \geq t}.
\]

In our application, $\mat{X}$ is the empirical covariance matrix and $\mat{A}$ is the actual covariance matrix. The spectral norm dominates the maximum eigenvalue, so 
\begin{align*}
 \Prob{\lambda_k(\widehat{\mat{C}}_n) \geq \lambda_k(\mat{C}) + t} & \leq \Prob{\lambdamax{\mat{W}_+^\star(\widehat{\mat{C}}_n - \mat{C})\mat{W}_+} \geq t}\\
& \leq \Prob{\|\mat{W}_+^\star (\widehat{\mat{C}}_n - \mat{C}) \mat{W}_+\| \geq t} = \Prob{\|\mat{W}_+^\star \widehat{\mat{C}}_n \mat{W}_+ - \mat{S}^2 \| \geq t },
\end{align*}
where $\mat{S}$ is the square root of $\mat{W}_+^\star \mat{C} \mat{W}_+.$ Now factor out $\mat{S}^2$ and identify $\lambda_k(\mat{C}) = \|\mat{S}^2\|$ to obtain
\begin{align*}
 \Prob{\lambda_k(\widehat{\mat{C}}) \geq \lambda_k(\mat{C}) + t} & \leq \Prob{\|\mat{S}^{-1} \mat{W}_+^\star \widehat{\mat{C}}_n \mat{W}_+ \mat{S}^{-1} - \mathbf{I} \| \|\mat{S}^2\| \geq t } \\
 & = \Prob{\|\mat{S}^{-1} \mat{W}_+^\star \widehat{\mat{C}}_n \mat{W}_+ \mat{S}^{-1} - \mathbf{I} \| \geq t/\lambda_k(\mat{C})}.
\end{align*}
Note that if $\vec{\eta}$ is drawn from a $\mathcal{N}(\vec{0}, \mat{C})$ distribution, then the covariance matrix of the transformed sample $\mat{S}^{-1}\mat{W}_+^\star \vec{\eta}$ is the identity:
\[
\E \left(\mat{S}^{-1}\mat{W}_+^\star \vec{\eta} \vec{\eta}^\star \mat{W}_+ \mat{S}^{-1}\right) = \mat{S}^{-1} \mat{W}_+^\star \mat{C} \mat{W}_+ \mat{S}^{-1} = \mathbf{I}.
\]
Thus $\mat{S}^{-1} \mat{W}_+^\star \widehat{\mat{C}}_n \mat{W}_+ \mat{S}^{-1}$ is the empirical covariance matrix of a standard Gaussian vector in $\R^{p-k+1}.$ By Theorem 1 of \cite{ALPT10b}, it follows that $\hat{\lambda}_k$ is unlikely to overestimate $\lambda_k$ in relative error when the number $n$ of samples is $\Omega(p-k+1).$ A similar argument shows that $\hat{\lambda}_k$ is unlikely to underestimate $\lambda_k$ in relative error when $n = \Omega(\kappa_p^2 k).$

Similarly, for more general distributions, the bounds on the probability of $\hat{\lambda}_k$ overestimating or underestimating $\lambda_k$ can be tightened beyond those suggested in Theorem \ref{thm:covarest} by using the results in \cite{ALPT10b} or \cite{V10}. Note, however, that one cannot use knowledge of spectral decay to sharpen the results obtained from \cite{ALPT10b} and \cite{V10} into estimates like those given in Theorem \ref{thm:examplecovarest}.

Finally, we note that the techniques developed in the proof of Theorem \ref{thm:covarest} can be used to investigate the spectrum of the error matrices $\widehat{\mat{C}}_n - \mat{C}.$
\subsection{Proofs of the supporting lemmas}

We now establish the lemmas used in the proof of Theorem \ref{thm:covarest}.

\begin{proof}[Proof of Lemma \ref{lemma:splittail}]
The probability that $\lambda_k(\mat{X})$ overestimates $\lambda_k(\mat{A})$ is controlled with the sequence of inequalities
\begin{align*}
\Prob{\lambda_k(\mat{X}) \geq \lambda_k(\mat{A}) + t} & = \Prob{ \inf_{\mat{W} \in \Isom{p-k+1}{p}} \lambdamax{\mat{W}^*\mat{X}\mat{W}} \geq \lambda_k(\mat{A}) + t} \\
& \leq \Prob{\lambdamax{\mat{W}_+^*\mat{X}\mat{W}_+} \geq \lambda_k(\mat{A}) + t}. 
\end{align*}

We use a related approach to study the probability that $\lambda_k(\mat{X})$ underestimates $\lambda_k(\mat{A}).$ Our choice of $\mat{W}_-$ implies that 
\begin{align*}
\Prob{\lambda_k(\mat{X}) \leq \lambda_k(\mat{A}) - t} & = \Prob{\max_{\mat{W} \in \Isom{k}{p}} \lambdamin{ \mat{W}^\star \mat{X} \mat{W}} \leq \lambda_k(\mat{A}) - t} \\
& \leq \Prob{\lambdamin{\mat{W}_-^*\mat{X}\mat{W}_-} \leq  \lambda_k(\mat{A}) - t } \\
& = \Prob{\lambdamax{\mat{W}_-^\star (-\mat{X}) \mat{W}_-} \geq - \lambda_k(\mat{A}) + t}.
\end{align*}
This establishes the bounds on the probabilities of $\lambda_k(\mat{X})$ deviating above or below $\lambda_k(\mat{A}).$
\end{proof}

\begin{proof}[Proof of Lemma \ref{lemma:momentbound}]

Factor the covariance matrix of $\vec{\xi}$ as $\mat{G} = \mat{U\Lambda U}^\star$ where $\mat{U}$ is orthogonal and $\mat{\Lambda} = \text{diag}(\lambda_1, \ldots, \lambda_p)$ is the matrix of eigenvalues of $\mat{G}$. Let $\vec{\gamma}$ be a $\mathcal{N}(\vec{0},\mathbf{I}_p)$ random variable. Then $\vec{\xi}$ and $\mat{U\Lambda}^{1/2} \vec{\gamma}$ are identically distributed, so 
\begin{align} 
\E (\vec{\xi}\vec{\xi}^\star)^m & = \E\left[(\vec{\xi}^\star\vec{\xi})^{m-1} \vec{\xi}\vec{\xi}^\star \right] = \E\left[ (\vec{\gamma}^\star \mat{\Lambda} \vec{\gamma})^{m-1} \mat{U\Lambda}^{1/2} \vec{\gamma}\vec{\gamma}^\star \mat{\Lambda}^{1/2} \mat{U}^\star \right]\notag \\
 & = \mat{U\Lambda}^{1/2} \E \left[ (\vec{\gamma}^\star \mat{\Lambda} \vec{\gamma})^{m-1} \vec{\gamma}\vec{\gamma}^\star \right] \mat{\Lambda}^{1/2}\mat{U}^\star. 
\label{eqn:rankonewishartpowers}
\end{align}
Consider the $(i,j)$ entry of the bracketed matrix in \eqref{eqn:rankonewishartpowers}:
\begin{equation}
\E\left[(\vec{\gamma}^\star \mat{\Lambda} \vec{\gamma})^{m-1} \gamma_i\gamma_j \right] = \E\left[\left(\sum\nolimits_{\ell=1}^p \lambda_\ell \gamma_\ell^2 \right)^{m-1} \gamma_i \gamma_j\right]. 
\label{eqn:gaussianchaos}
\end{equation}
From this expression, and the independence of the Gaussian variables $\{ \gamma_i \},$ we see that this matrix is diagonal.

To bound the diagonal entries, use a multinomial expansion to further develop the sum in \eqref{eqn:gaussianchaos} for the $(i,i)$ entry:
\[
 \E\left[(\vec{\gamma}^\star \mat{\Lambda} \vec{\gamma})^{m-1} \gamma_i^2 \right]  
= \sum_{\ell_1 + \cdots +\ell_p = m-1} \binom{m-1}{\ell_1, \ldots, \ell_p} \lambda_1^{\ell_1} \cdots \lambda_p^{\ell_p} \E\left[\gamma_1^{2\ell_1} \cdots \gamma_p^{2\ell_p} \gamma_i^2 \right].
\]
Denote the $L_r$ norm of a random variable $X$ by
\[ 
  \lnorm{r}{X} = \left(\E |X|^r \right)^{1/r}. 
\]
Since $\ell_1, \ldots, \ell_p$ are nonnegative integers summing to $m-1$, the generalized AM-GM inequality justifies the first of the following inequalities:
\begin{align*}
\E \gamma_1^{2\ell_1} \cdots \gamma_p^{2\ell_p} \gamma_i^2 & \leq  \E \left(\frac{\ell_1 |\gamma_1| + \cdots + \ell_p |\gamma_p| + |\gamma_i|}{m}\right)^{2m} =
 \lnorm{2m}{ \frac{1}{m} \left(|\gamma_i| + \sum_{j=1}^p \ell_j |\gamma_j| \right) }^{2m} \\
 & \leq \left( \frac{1}{m} \left( \lnorm{2m}{\gamma_i} + \sum_{j=1}^p \ell_j \lnorm{2m}{\gamma_j} \right) \right)^{2m} \\
 & = \left(\frac{1 + \ell_1 + \ldots + \ell_p}{m} \right)^{2m} \lnorm{2m}{g}^{2m} = \E (g^{2m}).  
\end{align*}
The second inequality is the triangle inequality for $L_r$ norms. Now we reverse the multinomial expansion to see that the diagonal terms satisfy the inequality
\begin{align}
\E\left[(\vec{\gamma}^\star \mat{\Lambda} \vec{\gamma})^{m-1}\gamma_i^2\right] & \leq \sum_{\ell_1 + \cdots + \ell_p = m-1} \binom{m-1}{\ell_1, \ldots, \ell_p} \lambda_1^{\ell_1} \cdots \lambda_p^{\ell_p} \E (g^{2m}) \notag \\
&  = (\lambda_1 + \ldots + \lambda_p)^{m-1}\E ( g^{2m}) = \tr(\mat{G})^{m-1} \E (g^{2m}). 
\label{eqn:diagonalest}
\end{align}
Estimate $\E(g^{2m})$ using the fact that $\Gamma(x)$ is increasing for $x \geq 1:$
\begin{align*}
\E \left( g^{2m} \right) & = \frac{2^m}{\sqrt{\pi}} \Gamma(m + 1/2) < \frac{2^m}{\sqrt{\pi}} \Gamma(m+1) = \frac{2^m}{\sqrt{\pi}} m! \quad \text{for } m \geq 1.
\end{align*}
Combine this result with \eqref{eqn:diagonalest} to see that
\[
\E\left[(\vec{\gamma}^\star \mat{\Lambda} \vec{\gamma})^{m-1}\vec{\gamma}\vec{\gamma}^\star \right] \preceq \frac{2^m}{\sqrt{\pi}} m! \tr(\mat{G})^{m-1} \cdot \mathbf{I}. 
\] 
Complete the proof by using this estimate in \eqref{eqn:rankonewishartpowers}.

\end{proof}


\bibliographystyle{amsalpha}
\bibliography{minimax_matrix_laplace_transform}
\end{document}